\documentclass[12pt,oneside,reqno]{amsart}
\usepackage{graphicx}
\usepackage{dcolumn}
\usepackage{amsmath,amsthm,amssymb}

\newtheorem{theorem}{Theorem}
\theoremstyle{remark}
\newtheorem{rem}{Remark}

\begin{document}

\title{Nonexistence of shrinkers for the harmonic map flow in higher dimensions}
\author{Piotr Bizo\'n}
\address{Institute of Physics, Jagiellonian
University, Krak\'ow, Poland\\ and
Max Planck Institute for Gravitational Physics (Albert Einstein Institute),
Golm, Germany}
\email{piotr.bizon@aei.mpg.de}

\author{Arthur Wasserman}
\address{Department of Mathematics, University of Michigan, Ann Arbor, MI 48109, USA}
\email{awass@umich.edu}
\thanks{The research of P.B. was supported in part by the NCN grant NN202 030740}%
\date{\today}%
\begin{abstract}
We prove that the harmonic map flow from the Euclidean space $\mathbb{R}^d$ into the sphere $S^d$ has no  equivariant self-similar  shrinking solutions in dimensions $d\geq 7$.
\end{abstract}
\maketitle

This note is concerned with the harmonic map flow for maps $u$  from the Euclidean space $\mathbb{R}^d$ to the sphere $S^d$. This flow, defined as  the gradient flow for the Dirichlet energy,
\begin{equation}\label{energy}
  E(u) = \int_{\mathbb{R}^d} |\nabla u|^2\,,
\end{equation}
 obeys the nonlinear heat equation
\begin{equation}\label{eqheat}
  u_t = \Delta u + |\nabla u|^2 u\,,
\end{equation}
where $(t,x)\in  \mathbb{R} \times \mathbb{R}^d$ and $u(t,x)\in S^d \hookrightarrow \mathbb{R}^{d+1}$.
 This equation is scale invariant: if $u(t,x)$ is a solution, so is $u_{\lambda}(t,x)=u(t/\lambda^2,x/\lambda)$. Under this scaling $E(u_{\lambda})= \lambda^{d-2} E(u)$ which means that $d=2$ is the critical dimension and higher dimensions are supercritical.

 We consider equivariant maps of the form (where $r=|x|$)
\begin{equation}\label{ansatz}
  u(t,x)=\left(\frac{x}{r}\,\sin{v(t,r)},\cos{v(t,r)}\right)\,.
\end{equation}
This symmetry assumption  reduces Eq.\eqref{eqheat} to the scalar heat equation
\begin{equation}\label{eqv}
  v_t = v_{rr} +\frac{d-1}{r}\, v_r -\frac{d-1}{2 r^2}\, \sin(2v)\,.
\end{equation}
A natural question, important for understanding the global behavior of solutions and formation of singularities, is whether there exist solutions of Eq.\eqref{eqv} which are invariant under scaling, i.e. $v(t/\lambda^2,r/\lambda)=v(t,r)$.
 Such self-similar solutions come in two kinds: self-similar expanding solutions (expanders for short) of the form
\begin{equation}
v(t,r)=g\left(\frac{r}{\sqrt{t}}\right), \quad  t>0\,,
\end{equation}
and
self-similar shrinking solutions (shrinkers for short) of the form
\begin{equation}\label{shrinker}
v(t,r)=f\left(\frac{r}{\sqrt{-t}}\right), \quad  t<0\,.
\end{equation}
Expanders are easy to construct in any dimension and well understood (see \cite{gr} and \cite{bb}) so here we will consider only shrinkers.  Substituting the ansatz \eqref{shrinker} into Eq.\eqref{eqv} and using the similarity variable $y=r/\sqrt{-t}$  we get an ordinary differential equation for $f(y)$ on $y\geq 0$
\begin{equation}\label{eqf}
   f''+\left(\frac{d-1}{y}-\frac{y}{2}\right)\,f'-\frac{d-1}{2y^2}\,\sin(2f)=0\,.
\end{equation}
It is routine to show that both near the center and near infinity there exist  one-parameter families of local smooth solutions satisfying
\begin{equation}\label{bc0}
  f(0)=0\,,\quad f'(0)=a>0\,,
\end{equation}
and
\begin{equation}\label{bcinf}
  f(\infty)=b\,,\quad \lim_{y\rightarrow \infty} y^3 f'(y)=-(d-1) \sin(2b)\,,
\end{equation}
where $a$ and $b$ are free parameters. The assumption that $a>0$ is made for later convenience (without loss of generality).
The question is: do there exist global smooth solutions satisfying the conditions \eqref{bc0} and \eqref{bcinf}? This question has been answered in affirmative for $3\leq d \leq 6$ by Fan \cite{fan}. Using a shooting method, Fan proved that there exists a countable sequence of pairs $(a_n,b_n)$ for which the local solutions satisfying \eqref{bc0} and \eqref{bcinf}
are smoothly connected by a globally regular solution $f_n(y)$. The positive integer $n$ denotes the number of intersections of the solution $f_n(y)$ with $\pi/2$. More detailed quantitative properties of the shrinkers were studied in \cite{bb}.

\begin{rem} To justify the conditions \eqref{bc0} and \eqref{bcinf}, let us recall that singularities of the harmonic map flow have been divided by Struwe \cite{s1}  into two types depending on whether  the quantity $(-t) |\nabla u|^2$ remains bounded (type I) or not (type II) as $t\nearrow 0$ (here we assume without loss of generality that the blowup occurs at time $t=0$). Calculating this quantity for the equivariant ansatz \eqref{ansatz} and \eqref{shrinker}  one finds that the blowup is of type I if and only if
\begin{equation}\label{type1}
 f'(y)^2+\frac{d-1}{y^2}\,\sin^2{\!f(y)} < C
\end{equation}
for some constant $C$ and all $y\geq 0$. The condition \eqref{type1} together with the requirement of smoothness is equivalent to
the conditions \eqref{bc0} and \eqref{bcinf}. In the case of \eqref{bc0} this is evident. To see how \eqref{bcinf} comes about, let us rewrite Eq.\eqref{eqf} in the integral form
\begin{equation*}\label{hop}
 f'(y)= \frac{d-1}{2} y^{1-d} e^{y^2/4} A(y), \quad  A(y)=\int_0^y s^{d-3} e^{-s^2/4} \sin(2f(s)) ds\,.
\end{equation*}
 For $f'(y)$ to be bounded at infinity, it is necessary that $\lim_{y\rightarrow\infty} A(y)= 0$ and  then   \eqref{bcinf} follows from  l'H\^{o}pital's rule. Thus, the conditions \eqref{bc0} and \eqref{bcinf} are equivalent to the requirement that the blowup mediated by the shrinker \eqref{shrinker} is of type I.
\end{rem}

 One of the key ingredients of the shooting argument in \cite{fan} is that the linearized perturbations about the equator map $f=\pi/2$ are oscillating at infinity. This happens for $d^2-8d+8<0$ which implies the upper bound $d=6= \left \lfloor{4+2\sqrt{2}}\right \rfloor $ (of course, only integer values of $d$ make sense geometrically). There is numerical evidence that there are no smooth shrinkers for $d\geq 7$, however to our knowledge this fact has not been proved. The aim of this note is to fill this gap by  proving the following non-existence result:
 \begin{theorem}
 For $d\geq 7$ there exists no  smooth solution of equation \eqref{eqf}  satisfying the conditions \eqref{bc0} and \eqref{bcinf}.
 \end{theorem}
 \begin{proof}
  The proof is extremely simple. Suppose that $f(y)$ is a global solution satisfying \eqref{bc0} and define
the function $h(y)=y^3 f'(y)$.
Multiplying  Eq.\eqref{eqf} by $y^2$ and differentiating we get
\begin{equation}\label{eqh}
  y^2 h''=\alpha(y) h' + \beta(y) h\,,
\end{equation}
where
\begin{equation}\label{ab}
   \alpha(y)=\frac{1}{2} y (y^2-2d+10)\quad \mbox{and}\quad \beta(y)=d-7+(d-1)(1+\cos{2f})\,.
\end{equation}
We assume that $d\geq 7$, so $\beta(y) \geq 0$.
 It follows from \eqref{bc0} that $h(0)=h'(0)=h''(0)=0$ and $h'''(0)=6a>0$, hence $h'(y)>0$ for small $y$.
   We now show that $h'(y)$ cannot go to zero. Suppose otherwise and let $y_0$ be the first point at which $h'(y_0)=0$. If $d>7$ or $f(y_0)\neq \pi/2$, then $\beta(y_0)>0$ and therefore $h''(y_0)=\beta(y_0) h(y_0)>0$, contradicting that $y_0$ exists. If $d=7$ and $f(y_0)= \pi/2$, then $\beta(y_0)=0$ and $h''(y_0)=0$, so a bit more work is needed. In this case,
  differentiating Eq.\eqref{eqh} we find that $h'''(y_0)=0$ and differentiating once more we get $ h^{(iv)}(y_0)=24 y_0^{-8} \, h^3(y_0)>0$, again contradicting the existence of $y_0$.
   Thus $h'(y)>0$ for all $y$. From this, \eqref{eqh}, and \eqref{ab}  we obtain
\begin{equation}\label{hbis}
  y^2 h''(y)=\alpha(y) h'(y)+ \beta(y) h(y)>0\quad \mbox{for}\,\,y> \sqrt{2d-10}\,.
\end{equation}
 Therefore, $\lim_{y\rightarrow \infty} h(y)=\infty$, contradicting \eqref{bcinf}.
\end{proof}
We conclude with a few remarks.
\begin{rem}
Eq.\eqref{eqf} is the Euler-Lagrange equation for the functional
\begin{equation}\label{e}
    \mathcal{E}(f)=\int_0^{\infty}  \left(f'^2+\frac{d-1}{y^2} \sin^2\!{f} \right)\, y^{d-1}\, e^{-y^2/4}\, dy\,,
\end{equation}
which can be interpreted as the Dirichlet energy for  maps from $\mathbb{R}^d$ with the conformally flat metric $g=e^{-\frac{y^2}{2(d-2)}} g_{flat}$ into $S^d$. Thus, shrinkers can be viewed as harmonic maps from $(\mathbb{R}^d,g)$ into $S^d$.  Note that  $\mathcal{E}(f)$ is invariant under the reflection symmetry $f\to \pi-f$ and the equator map $f_e=\pi/2$ is the only fixed point of this symmetry. For this kind of functional
Corlette and Wald conjectured in \cite{cw}, using Morse theory arguments,  that the number of  critical points (counted without multiplicity) with energy below $\mathcal{E}(f_e)$ is equal to the Morse index of $f_e$ (i.e. the number of negative eigenvalues of the Hessian of $\mathcal{E}$ at $f_e$).  In the case of \eqref{e}, the Morse index of $f_e$  drops from infinity to two at $d=4+2\sqrt{2}$ and then from two to one at $d=7$ (see \cite{b}). Thus, according to the conjecture of Corlette and Wald, for $d\geq 7$ there should be exactly one (modulo the reflection symmetry) critical point of $\mathcal{E}(f)$ (this unique critical point is, of course, $f=0$), in perfect agreement with Theorem~1.
\end{rem}
\begin{rem}
Struwe showed that the type I singularities are asymptotically self-similar \cite{s2}, that is their profile is given by a smooth shrinker.
Therefore, Theorem~1 implies that in dimensions $d\geq 7$ all singularties for the equivariant harmonic map flow \eqref{eqv} must be of type II (see \cite{b} for a recent analysis of such singularities).
\end{rem}
\begin{rem}
It is well-known that there are close parallels between the harmonic map and  Yang-Mills heat flows \cite{g2}. For the spherically symmetric magnetic Yang-Mills potential $w(t,r)$ in $d\geq 3$ dimensions  a counterpart of Eq.\eqref{eqv} reads
     \begin{equation}\label{ym}
  w_t=  w_{rr} +\frac{d-3}{r}\, w_r- \frac{d-2}{r^2}\, w (w-1) (w-2)\,,
\end{equation}
and a counterpart of Eq.\eqref{eqf} for shrinkers $w(t,r)=g(y)$ is
\begin{equation}\label{eqg}
   g''+\left(\frac{d-3}{y}-\frac{y}{2}\right)\,g'-\frac{d-2}{y^2}\,g (g-1) (g-2)=0\,.
\end{equation}
The one-parameter families of local smooth solutions of this equation near the origin and near infinity satisfy
\begin{equation}\label{bc0ym}
  g(0)=g'(0)=0\,,\quad g''(0)=a>0\,,
\end{equation}
and
\begin{equation}\label{bcinfym}
  g(\infty)=b\,,\quad \lim_{y\rightarrow \infty} y^3 g'(y)=-2(d-2) b (b-1) (b-2)\,.
\end{equation}
Using a similar shooting technique as in \cite{fan} one can easily show that for $5\leq d\leq 9$ there are infinitely many shrinkers $g_n(y)$. One novel feature, in comparison with the harmonic map flow, is that the first  shrinker is known explicitly \cite{w}:
\begin{equation}\label{g1}
g_1(y)=\frac{y^2}{\gamma+\delta y^2},\quad
\gamma=\frac{1}{2}(6d-12-(d+2)\sqrt{2d-4})\,,\quad \delta=\frac{\sqrt{d-2}}{2\sqrt{2}}\,.
\end{equation}
In complete analogy to Theorem~1 we have:
\begin{theorem}
 For $d\geq 10$ there exists no smooth solution of equation \eqref{eqg}  satisfying the conditions \eqref{bc0ym} and \eqref{bcinfym}.
 \end{theorem}
 \begin{proof}
 The same as for Theorem~1. The only change is that now the function $h(y)=y^3 g'(y)$ satisfies Eq.\eqref{eqh} with different coefficients
 \begin{equation}\label{abym}
   \alpha(y)=\frac{1}{2} y (y^2-2d+14)\quad \mbox{and}\quad \beta(y)=d-10+3(d-2)(1-g)^2\,.
\end{equation}
Note that for $d=10$ the solution \eqref{g1} becomes $g_1=1$ (which does not satisfy the regularity condition at the origin \eqref{bc0ym}), while for $d>10$ the parameter $\gamma$ is negative so $g_1(y)$ has a pole at $y=(-\gamma/\delta)^{1/2}$.
 \end{proof}
\end{rem}
By  arguments analogous to the ones given  in Remarks~1 and~3, it follows that  in dimensions $d\geq 10$ all singularities for the equivariant Yang-Mills flow \eqref{ym} must be of type II.

\end{document}